\documentclass[10pt,oneside]{amsart} %added oneside so comments stay on the right, can remove when finished

\usepackage{amsmath,amstext,amscd, amssymb, color, mathtools,mleftright} %txfonts, 
\usepackage{pinlabel}
\usepackage[normalem]{ulem}
\usepackage[colorlinks=true, pdfstartview=FitV, linkcolor=blue, citecolor=blue, urlcolor=blue]{hyperref}
 \usepackage[abs]{overpic}
\usepackage{xcolor,varwidth}
\usepackage{enumerate, scalerel}

\makeatletter
\def\@tocline#1#2#3#4#5#6#7{\relax
\ifnum #1>\c@tocdepth % then omitf
  \else 
    \par \addpenalty\@secpenalty\addvspace{#2}% 
\begingroup \hyphenpenalty\@M
    \@ifempty{#4}{%
      \@tempdima\csname r@tocindent\number#1\endcsname\relax
 }{%
   \@tempdima#4\relax
 }%
 \parindent\z@ \leftskip#3\relax \advance\leftskip\@tempdima\relax
 \rightskip\@pnumwidth plus4em \parfillskip-\@pnumwidth
 #5\leavevmode\hskip-\@tempdima #6\nobreak\relax
 \ifnum#1<0\hfill\else\dotfill\fi\hbox to\@pnumwidth{\@tocpagenum{#7}}\par
 \nobreak
 \endgroup
  \fi}
\makeatother
\let\oldtocsection=\tocsection
\let\oldtocsubsection=\tocsubsection
\let\oldtocsubsubsection=\tocsubsubsection
\renewcommand{\tocsection}[2]{\hspace{0em}\oldtocsection{#1}{#2}}
\renewcommand{\tocsubsection}[2]{\hspace{1em}\oldtocsubsection{#1}{#2}}
\renewcommand{\tocsubsubsection}[2]{\hspace{2em}\oldtocsubsubsection{#1}{#2}}
\definecolor{cerulean}{rgb}{0,.48,.65} 
\definecolor{magenta}{rgb}{.5,0,.5} 
\definecolor{dred}{rgb}{.5,0,0} 
\definecolor{green}{rgb}{0,.5,0} 
\definecolor{blue}{rgb}{0,0,1} 
\definecolor{black}{rgb}{0,0,0} 
\definecolor{dgreen}{rgb}{0,.3,0} 
\definecolor{vdred}{rgb}{.3,0,0} 
\definecolor{red}{rgb}{1,0,0} 
\definecolor{salmon}{rgb}{0.98,0.50,0.45} 
\definecolor{gray}{rgb}{.5,.5,.5} 
\definecolor{seagreen}{rgb}{0.13,0.70,0.67} 
\definecolor{chartreuse}{rgb}{0.40,0.80,0.00}
\definecolor{cornflower}{rgb}{0.39,0.58,0.93} 
\definecolor{gold}{rgb}{0.65,0.45,0.00}
 
\theoremstyle{plain}

\newtheorem{theorem}{Theorem}
\newtheorem{thm}{Theorem}[section]
\newtheorem{lemma}[thm]{Lemma}

\newtheorem{cor}[thm]{Corollary}
\newtheorem{prop}[thm]{Proposition}

\newtheorem*{conjproblem*}{The conjugacy problem}
\newtheorem*{0twistedproblem*}{The 0-twisted-conjugacy problem}
\newtheorem*{Htwistedproblem*}{The H-twisted conjugacy problem}
\newtheorem*{Itwistedproblem*}{The I-twisted conjugacy problem}

\theoremstyle{definition}

\newtheorem{remark}[thm]{Remark}

\newtheorem{Open questions}[thm]{Open questions}
\newtheorem{Open question}[thm]{Open question}
\newtheorem{Open problems}[thm]{Open problems}
\newtheorem{Open problem}[thm]{Open problem}

\def\Bbb{\mathbb}

\def\Z{\Bbb{Z}}
\def\N{\Bbb{N}}

\def\ni{\noindent}

\def\Dist{\hbox{\rm Dist}}

\def\Dehn{\hbox{\rm Dehn}}

\def\Dehn{\hbox{\rm Dehn}}

\def\CL{\hbox{\rm CL}}

\def\F+L{\hbox{$\textup{F}\!_+\textup{L}$}}

\def\ssm{\smallsetminus}

\DeclareMathOperator{\lcm}{lcm}

\newtheorem*{thmI}{Theorem $1'$}

\def\onto{{\kern3pt\to\kern-8pt\to\kern3pt}}

\def\<{\langle}
\def\>{\rangle}
\def\|{{\ |\ }}

\def\g{\gamma}
\def\G{\Gamma}

\newcommand{\set}[1]{\left\{#1\right\}}

\newcommand{\abs}[1]{\left|#1\right|}
\renewcommand{\ni}{\noindent}

\def\*{^{\star}}

\renewcommand{\o}[1]{\overline{#1}}

\begin{document}

\title[Fast growing conjugator length functions]{Groups with fast-growing \\ conjugator length functions} 
\author{M.\ R.\ Bridson and T.\ R.\ Riley}

\date{3 August 2026}

\begin{abstract}
\ni    We construct the first examples of finitely presented groups whose conjugator length function is exponential; these
are central extensions of groups of the form $F_m\rtimes F_2$.
Further, we use a fibre product construction to exhibit a family of finitely presented groups $\Gamma_k$
where, for each  $k$,
the conjugator length function of $\Gamma_k$ grows like functions lying in the $k$-th level of the Grzegorczyk hierarchy of primitive recursive functions.
 
  \smallskip
%\\
\ni \footnotesize{\textbf{2020 Mathematics Subject Classification:  20F65, 20F10}}  \\ 
\ni \footnotesize{\emph{Key words and phrases:} conjugacy problem,  conjugator length}
\end{abstract}

\thanks{We gratefully acknowledge the financial support of the Clay Mathematics Institute (MRB) and the Simons Foundation (TRR--Simons Collaboration Grant 318301) and the National Science Foundation (TRR, NSF GCR-2428489). ORCID: 0000-0002-0080-9059 (MRB),  0009-0004-3699-0322 (TRR)}

\maketitle

%\setcounter{tocdepth}{2}
%\tableofcontents 

\section{Introduction} \label{intro}

This is one of a series of articles in which we explore the geometry of the conjugacy problem in finitely presented groups
through the lens of conjugator length functions.  Our purpose in this article is to construct explicit groups where
these functions exhibit a range of rapid growth types.
 Given a group $\Gamma$ with a finite generating set $A$, we write $u \! \sim \!  v$ when words $u$ and $v$ in $A^{\pm 1}$ represent conjugate  elements  of $\Gamma$,  and   define $\CL(u,v)$ to be
 the length of a shortest word $w$ such that $uw=wv$ in $\Gamma$.
 The \emph{conjugator length function} $\CL : \N \to \N$ is defined so that $\CL\mleft(n\mright)$ is the least
 integer  $N$ such that $\CL(u,v) \leq N$ for all words $u$ and $v$ with lengths $|u|+|v|\le n$ such that
  $u \! \sim v$ in $\Gamma$.

We proved in \cite{BrRi2, BrRi1} that for all $d \in \N$, there are finitely presented groups for which   $\CL\mleft(n\mright) \simeq n^d$, and   in \cite{BrRi4} we proved that the set of exponents $e$ for which there is a finitely presented group with $\CL\mleft(n\mright) \simeq n^e$ is dense in $[2, \infty)$.  Many constructions
of finitely presented groups are known that have a solvable word problem but an unsolvable conjugacy problem (e.g.~\cite{miller1}), 
and the  conjugator  length function of such a group  is not bounded above by any recursive function.   But it has proved
surprisingly difficult to construct finitely presented groups with matching upper and lower bounds that are large. 

The bulk of this article is dedicated to constructing the  first examples of  finitely presented groups where the conjugator length function is exponential.
 
\begin{theorem} \label{t: exponential CL}
There exist finitely presented groups $\Lambda$ such that $\CL\mleft(n\mright) \simeq 2^n$,  for instance, 
\begin{align*} 
\Lambda  & = \ \scaleleftright[1.75ex]{\bigg\langle}{     \, a_1, a_2, a_3, s, t,  \lambda \  \left| \ \parbox{69mm}{$s^{-1} a_1 s = a_{2}, \ s^{-1} a_2 s = a_{3}, \  s^{-1} a_3 s = a_1 a_2 a_3,$ \\  $[t,a_1] =  [t,a_2]  = [t,a_3] = \lambda,$ \\  $[a_1,\lambda] = [a_2,\lambda] = [a_3,\lambda] = [s,\lambda] = [t,\lambda] = 1$} \, \right. }{\bigg\rangle}.
\end{align*} 
 \end{theorem}

We also construct  examples where the growth of the conjugator length function is
 comparable to the fast-growing functions $A_k$, which constitute the Ackermann function $(k,n) \mapsto A_k(n)$ and represent the successive levels of the Grzegorczyk hierarchy of primitive recursive functions: 
$A_0\mleft(n\mright)= n+2,\ A_1\mleft(0\mright)=0,\ A_k\mleft(0\mright)=2$ for $k\ge 2$, and for all $k,n\ge 0$,
$$
A_{k+1}\mleft(n + 1\mright) = A_k\mleft(A_{k+1}\mleft(n\mright)\mright).  
$$ 
So $A_1\mleft(n\mright) = 2n$, $A_2\mleft(n\mright) = 2^n$, and $A_3\mleft(n\mright)$ is a height-$n$ tower of powers of $2$.  (See, for example, \cite{rose}.)

\begin{theorem} \label{t:big fellas}
For all $k$, there is a finitely presentable group $\G$ whose conjugator length function satisfies $A_k\mleft(n\mright) \preceq \CL_\Gamma\mleft(n\mright) \preceq A_k\mleft(n^2\mright)$.  Thus %when $k \geq 3$, 
$\CL_\G\mleft(n\mright)$ is sandwiched  between a pair of fast-growing functions that are both in the $k$-th level of the Grzegorczyk hierarchy  that grades the primitive recursive functions.
\end{theorem}

Here are outlines of our constructions.

\subsection*{Groups with exponential conjugator length functions} 
The ideas behind our construction apply to large classes of groups, but  
in order to minimize the many technical details and to avoid becoming overwhelmed by notation,  
we focus on a family of central extensions of free-by-free groups.
   
Let $\phi$ be an atoroidal\footnote{meaning that no power of $\phi$ leaves a non-trivial
conjugacy class invariant} automorphism of the rank-$m$ free group
$F = F(a_1, \ldots, a_m)$ and define
\begin{align*}
H & = \langle a_1, \ldots, a_m, s \mid s^{-1} a_i s = \phi(a_i)  \  \   \forall i \  \rangle \\ 
G &  =     \langle \, H, t  \mid \left[t,a_1\right]=  \cdots  = \left[t,a_m\right]=1   \rangle \\ 
\Lambda^\phi &  =     \langle \, H, t, \lambda  \mid \left[t,a_1\right]=  \cdots  = \left[t,a_m\right]=\lambda, \  \lambda \textup{ central }   \rangle .
\end{align*} 
Then $H = F\rtimes_\phi \Z$ is a torsion-free hyperbolic group (by \cite{BF,  Brinkmann}),
$G$ is the trivial HNN-extension $H \dot\ast_F$ over $F<H$
(which decomposes as a semidirect product $F_m\rtimes\<s,t\>$),  
and $\Lambda^\phi$ is a  central extension of $G$ with centre $\<\lambda\>\cong\Z$.  (To see that  $\lambda\in\Lambda^\phi$
has infinite order, note that $\Lambda^\phi = (H\times\<\lambda\>)\rtimes \<s,t\>$.)  We shall restrict our attention to 
automorphisms $\phi$ that have {\em homological stretch}, meaning that the action of $\phi$ on the abelianisation of 
$F$ has an eigenvalue of absolute value greater than $1$; equivalently,  there is a constant $c>1$
so that for some $a_i$ the sum of the exponents on the letters in $\phi^n(a_i)$ is bounded below by $c^{|n|}$ for all $n\in\Z$.
Such stretch always occurs when the atoroidal automorphism $\phi$ is  {\em positive}, i.e.~when each $\phi\mleft(a_i\mright)$ is a word in $a_1, \ldots, a_m$ with only  positive exponents.

\begin{thmI} \label{Thm1'} If $\phi\in{\rm{Aut}}(F)$ is an atoroidal automorphism with homological stretch, then $\CL_{\Lambda^\phi} \mleft(n\mright)    \simeq    2^n.$
\end{thmI}

An example of a positive atoroidal automorphism   (from \cite[Example 3.2]{GS91}) is 
\begin{equation} \label{eq:phi}
\phi(a_i) = \begin{cases} a_{i+1} & \text{for } i=1, \ldots, m-1 \\    a_1 \cdots a_m & \text{for } i=m \end{cases}
\end{equation}  
where $m \geq 3$.   The group given in Theorem~\ref{t: exponential CL} is the $m=3$ instance of this.

The proof of Theorem~$1'$ occupies Sections~\ref{s: Bezout}--\ref{s: CL Lambda at least exp}.
Throughout the proof,   we will retain the notation  $F$,  $H$, and $G$, for the groups defined above.
We will normally abbreviate $\Lambda^\phi$ to $\Lambda$. 
 
If $a$ and $b$ are integers with greatest common divisor $d$,  then B\'ezout's Lemma provides integer solutions $(x,y)$
to the equation  $ax  + by = d$.
In Section~\ref{s: Bezout} we will quantify certain variations of this situation that are needed in our proof of the
upper bound $\CL_{\Lambda^\phi} \mleft(n\mright)    \preceq    2^n$.
  Such quantifications are a recurring theme in our papers on conjugator length, where  finding a short conjugator for a pair of group elements known to be conjugate hinges on finding a \emph{small} solution to a system of
  linear Diophantine equations that is known to have \emph{some} solution.

In Section~\ref{s:G} we discuss the structure of $G$ in relation to its hyperbolic subgroup $H$ and its free subgroups $F$ and $E = \langle s, t \rangle$ and we establish notation for the arguments that follow.  
In Section~\ref{s: distortion} we prove that the distortion of the central subgroup 
$\langle \lambda \rangle<\Lambda$ is exponential  $\Dist_{\langle \lambda \rangle}^{\Lambda}(n) \simeq 2^n$. The  upper
bound holds for quite general reasons but the lower bound is specific to the structure of $\Lambda$; it relies on the assumption
that $\phi$ has homological stretch.

In Section~\ref{s: conjugators in H and E} we discuss the length of  conjugators in the hyperbolic groups $H$ and $E$.
Centralisers of non-trivial elements in torsion-free hyperbolic groups are generated by maximal roots of those elements and, with the help of the \emph{uniformly monotone cyclics property} enjoyed by torsion-free hyperbolic groups, we will estimate the lengths of those roots.  In Section~\ref{s: centraliser and conjugators in G}  we will promote our understanding of element-centralisers in $H$ and $E$ to a  classification of element-centralisers in $G$.  The need for this classification arises naturally in our analysis of
 conjugators in $G$ because if $u \sim v$ in $G$ and one has hold of a specific $w_0$ 
 such that $uw_0=w_0v$ in $G$,  then $Z_G(u)w_0$ is the set of \emph{all} such conjugators. 
 From this analysis we will identify conjugators that are optimal in an appropriate sense.   A key result
 here is Lemma~\ref{lem:best w},  which provides bounds on conjugator lengths in $G$ that
 play a pivotal role in Section~\ref{s: CL in lambda}.

With this  understanding of element-centralisers in $G$ in hand,
in Section~\ref{s: CL in lambda} we complete the proof of the upper bound  $\CL_{\Lambda}\mleft(n\mright)   \preceq 2^n$ 
in Theorem~$1'$.  In outline, the argument proceeds as follows. 
Given words $\o{u}$ and $\o{v}$ representing conjugate elements in $\Lambda$,
we consider  the images  $u,v\in G$  of $\o{u}$ and $\o{v}$ and search through all 
words $W$ such that $uW =Wv$ in $G$---at least one will satisfy $\o{u}W = W\o{v}$ in $\Lambda$.  We will be able to access all the $W$ such that $uW =Wv$ in $G$ according to our classification of element-centralisers in $G$ from Section~\ref{s: centraliser and conjugators in G}, and we can tell which 
   are  conjugators for $\o{u}$ and $\o{v}$ in $\Lambda$ by calculating whether the  integer $N$ such that $\o{u}W = W\o{v} \lambda^N$ in $\Lambda$ is zero.    In some instances there is only one possible $W$.  In other cases, we will argue that one of the quantifications of B\'ezout from Section~\ref{s: Bezout} will guarantee that a sufficiently short $W$ exists among the possibilities. In every case, our estimates will give that there exists a conjugator of length within the $\preceq 2^n$  bound.   
  
  Finally, in Section~\ref{s: CL Lambda at least exp}  we complete the proof of Theorem~$1'$
  by exhibiting a family of  pairs of  words that are conjugate in $\Lambda$ but only via long conjugators, thus  establishing the lower bound $\CL_{\Lambda}\mleft(n\mright)   \succeq 2^n$.

 \subsection*{Fibre products and groups with fast growing conjugator length functions.}
  Our examples establishing Theorem~\ref{t:big fellas} are explained in Section~\ref{s:Fibre products}.  They come from applying a fibre-product construction from \cite{Bridson14} to the Hydra groups of Dison and Riley \cite{DR} -- these
  are groups $\Gamma_k$ with concise finite  presentations that have $A_k$ as their Dehn functions.  As we shall explain, these presentations are  aspherical, which is a  feature that is needed to derive
  a finite presentation for the fibre product $P$ that we consider.  The difficulty of the 
   conjugacy problem in $P$, as witnessed by its conjugator length function, 
   reflects the difficulty of  the word problem of the seed group $\Gamma$, as witnessed by its Dehn function.  The current state of knowledge does not make this relationship precise  enough to determine the conjugator length function of $P$ exactly in general.  However,  we can establish that $A_k\mleft(n\mright) \preceq \CL\mleft(n\mright) \preceq A_k(n^2)$ and this places $\CL\mleft(n\mright)$ between a pair of fast-growing representatives of the same level of the Grzegorczyk hierarchy  per Theorem~\ref{t:big fellas}, because for $k \geq 3$, if a function is in the same level as $A_k$, then precomposing that function with a polynomial does not change where in the hierarchy  that function lies.

We conclude  (in Remark~\ref{rem:frontier}) with comments on how a study of a variant of the 
Dehn function introduced in \cite{Bridson13} might lead to explicit examples of finitely presented groups realizing  a wider variety of conjugator length functions.\footnote{Update:\ using a different approach, Gillis and Wagner \cite{GiWa} recently proved definitive results describing which functions can arise.}

\subsection*{Acknowledgement}    This paper developed out of a long-running project on which Andrew~Sale  also worked.  We are grateful  to him for his insights and his companionship.    We also thank an anonymous referee for their comments.

 \section{Preliminaries} \label{preliminaries}

\subsection*{Words}  We write   $[x,y] :=x^{-1}y^{-1}xy$ and $x^y := y^{-1} x y$.   The length $|w|$ of a word $w= x_{i_1}^{\mu_1} \cdots x_{i_n}^{\mu_n}$ in an alphabet $A = \set{ a_1, \dots, a_n }^{\pm 1}$ is $|\mu_1| + \cdots + |\mu_n|$.  If $A$ is a generating set of a group $\G$ and $\g \in \G$, then $|\g|_\G$ denotes the length of a shortest word  representing $\g$.     

 \subsection*{Subgroup distortion}

For a  finitely generated subgroup $S$ of a finitely generated group $\G$,  
 the \emph{distortion function}  $\Dist^{\G}_{S} : \N \to \N$,  with respect to fixed finite generating sets,  is
$$\Dist^{\G}_{S}\mleft(n\mright) \ := \  \max \set{  \  \abs{\g}_S  \  \mid  \   \g \in {S} \textup{ with }  \abs{\g}_\G  \leq n   \ }.$$

\subsection*{Qualitative equivalence of growth rates}
We use the  relation $\simeq$ that is standard in geometric group theory: for  functions  $f,g: \mathbb{N} \to \mathbb{N}$ write $f \preceq g$ when there exists $C>0$ such that $f\mleft(n\mright) \leq Cg(  Cn+C ) +Cn +C$ for all $n \geq 0$, and  write $f  \simeq g$ when $f \preceq g$ and $g \preceq f$.  

For us, pertinent consequences of this definition include  (1) when $S \leq \G$, any choices of finite generating sets for $S$ and $\G$ give rise to equivalent subgroup distortion functions $\Dist^{\G}_{S}$; (2) the cost of cyclically permuting a word has no impact on the growth of the conjugator length function up to equivalence; and (3) any two finite generating sets for a group give rise to equivalent  conjugator length functions.

\section{Quantifications around  B\'ezout's Lemma} \label{s: Bezout}
 
The following elementary estimates pertaining to B\'ezout's Lemma play an important
role in our proof. 
Lemma~\ref{lemma Bezout} follows a similar account in our paper \cite{BrRi2}.  We add  Corollary~\ref{cor Bezout lcm} because it will lead to a bound on $|\sigma(w)|$ in Lemma~\ref{lem:best w} that is crucial in the final case of  our proof  that $\CL_{\Lambda}(n) \leq 2^n$ in Section~\ref{s: CL in lambda}; Corollary~\ref{cor Bezout multi} is also called upon multiple times in that proof.

\begin{lemma} \label{lemma Bezout}
Suppose $a$ and $b$ are non-zero integers and the linear Diophantine equation $ax+by=c$ has a solution $(x,y)$. Let $d = \textup{gcd}(a,b)$. Then $ax+by=c$ has a solution with  
$|x| \leq | \frac{b}{d}|$ and  $|y| \leq \max \set{|\frac{c}{b}|, |\frac{a}{d}|}$.  
\end{lemma}

\begin{proof}
Because $ax+by=c$ has a solution, $d = \gcd(a,b) $ divides $c$, and if $(x_0, y_0)$ is an integer solution,  then the set of all such solutions is $$\set{ \left. \left(x_0 +  \frac{b}{d} k , y_0 - \frac{a}{d} k\right) \, \right|  \, k \in \Z \, }.$$  Therefore there exists a solution $(x,y)$ with $0 \leq x < \frac{b}{d}$ and another with $\frac{b}{d} \leq x < 0$. Now $y = \frac{c}{b} - \frac{a}{b} x$, so $\frac{c}{b} - \frac{a}{d} < y \leq \frac{c}{b}$ for the first of these two solutions, and $ \frac{c}{b} < y \leq \frac{a}{d} +  \frac{c}{b}$ for the second.        
	
If  $\frac{a}{d}$ and $\frac{c}{b}$ have the same sign, then $|\frac{c}{b} - \frac{a}{d}| \leq \max   \set{|\frac{c}{b}|, |\frac{a}{d}|}$ and so the first of our two solutions satisfies the requirements of the lemma.  If they have opposite signs, then similarly the  second of our two solutions works.  
\end{proof}

\begin{cor} \label{cor Bezout lcm}
Under the hypotheses of Lemma~\ref{lemma Bezout},   $ax+by=c$  has an integer solution $(x,y)$ such that 
$|ax| \leq  \textup{lcm}(a,b)$ and  $|by| \leq \max \set{|c|, \textup{lcm}(a,b) }$.   	
\end{cor}

\begin{proof}
This follows  from  Lemma~\ref{lemma Bezout} because $\textup{lcm}(a,b) = ab / \textup{gcd}(a,b)$.
\end{proof}

\begin{cor} \label{cor Bezout multi}
Suppose the linear Diophantine equation 
\begin{equation} \label{eq mult linear diophantine}
a_1 x_1 +  \cdots + a_m x_m = c 
\end{equation} 
 has a solution $(x_1, \ldots, x_m)$.  Then it has a solution such that for all $i$,
 \begin{equation}  \label{eq mult diop bound}
|x_i| \leq \max \set{|a_1|, \ldots, |a_m|, |c|}.  
\end{equation}
\end{cor}

\begin{proof}
We prove this corollary by induction on $m$.  If $m=1$ or if $m=2$ and one of   $a_1$ and $a_2$ is zero, then the result is trivial.  If $m=2$ and $a_1$ and $a_2$ are non-zero, then the result follows from Lemma~\ref{lemma Bezout}.  Let $m \geq 3$.  Assume $a_1$, \ldots, $a_m$ are all non-zero, else the result holds by  induction.   

The integers expressible as $a_2 x_2 +  \cdots + a_m x_m$ for some $x_2, \ldots, x_m \in \Z$  are precisely the multiples of  $e := \gcd(a_2, \ldots, a_m)$.  So the hypothesis of the corollary tells us that the Diophantine equation $a_1x_1+ e y=c$ has a solution $(x_1, y)$,  and Lemma~\ref{lemma Bezout} allows us to assume
that $|x_1| \leq | \frac{e}{d}|$ and $|y| \leq \max \set{|\frac{c}{e}|, |\frac{a_1}{d}|}$, where $d = \gcd(a_1, e)$.          

Now, for some $x_2, \ldots, x_m \in \Z$ we have  $e y = a_2 x_2 +  \cdots + a_m x_m$ and so $y = \frac{a_2}{e} x_2 +  \cdots + \frac{a_m}{e} x_m$, which is a Diophantine equation because $e$ divides each of $a_2$, \ldots, $a_m$. By induction, there exist  $x_2, \ldots, x_m \in \Z$ satisfying this equation with $|x_i| \leq \max \set{ |\frac{a_2}{e}|, \ldots, |\frac{a_m}{e}|, |y|}$. This gives $x_1, \ldots, x_m \in \Z$ satisfying \eqref{eq mult linear diophantine}, and in combination with the above  bounds on $|x_1|$ and $|y|$, this establishes \eqref{eq mult diop bound}.      
\end{proof}

\section{The structure of $G$} \label{s:G}

By way of reminder, $F = F(a_1, \ldots, a_m)$ is a rank-$m$ free group, $H = F\rtimes_\phi \Z$ is hyperbolic,  and $G$ is the HNN-extension $H \dot\ast_F$ of $H$: 
\begin{align*}
H & = \langle a_1, \ldots, a_m, s \mid s^{-1} a_i s = \phi(a_i)  \  \   \forall i \  \rangle \\ 
G &  =     \langle \, H, t  \mid [t,a_1]=  \cdots  = [t,a_m]=1   \rangle .
\end{align*} 

In preparation for the sections that follow,  we shall describe some decompositions and retractions of
$G$ and establish useful notation.  

We will work with the fixed basis $\{ a_1, \ldots, a_m \}$ for $F$, and to generate $H$, $G$, and $\Lambda$ we add  $s$, then $t$, and then $\lambda$, respectively (paying attention to the fact that the role of $t$ in $\Lambda$ is different from its role in $G$).
Our generating set for the rank-2 free group $E$ is $\{ s, t \}$. For $g \in \Lambda$, we write $| g |_{\Lambda}$ for the length of a shortest word in our generating set for $\Lambda$ representing $g$.  We define $| g |_E$,  $| g |_F$, $| g |_G$, and $| g |_H$ likewise.    

We have the semidirect product decomposition  $G = F  \rtimes E$.
We view  $E$, $F$, and $H$ as subgroups of $G$ and will  use three retractions:  
\begin{itemize}
\item $G \onto H$ killing $t$, denoted by $g \mapsto g_H$, 
\item $G \onto E$ killing $a_1, \ldots, a_m$, denoted by $g \mapsto g_E$, and 
\item $G \onto \Z \cong \langle s \rangle$ killing $a_1, \ldots, a_m, t$ denoted by $g \mapsto \sigma(g)$.  
\end{itemize}

Because these maps $G \onto H$ and  $G \onto E$ are   defined by killing certain generators,  
 \begin{equation} \label{eq length under quotients} 
\forall \gamma \in G, \  \max\set{|\gamma_E |_E, | \gamma_H |_H}  \leq | \gamma |_G
\end{equation}   
and because they are retracts,   
\begin{equation} \label{eq length under retracts}
\forall \gamma \in H, \  | \gamma |_G = | \gamma |_H  \ \text{ and } \ \forall \gamma \in E, \  | \gamma |_G = | \gamma |_E.
\end{equation}

 According to the decomposition
  $G = F  \rtimes E$,  every $g \in G$ can be expressed uniquely as $g_F g_E$ for some $g_F \in F$ and the $g_E \in E$ that is defined above.  Because $t$ acts trivially on $F$, 
 \begin{equation} \label{acting via sigma} 
\forall f \in F, \ \forall x \in E, \ \ x^{-1} f x  = s^{-\sigma(x)} f s^{\sigma(x)} \text{ in } G.    
\end{equation}

Define $E_0 := \ker( \sigma : E \to \Z)$. This is the set of elements of $E$ that commute with every element of $F$.  

We extend the above notation to the level of words: if $w$ is a word in $a_1, \ldots, a_m, s, t$ (the generators of $G$), then $w_H$ is the word obtained from $w$ by deleting all letters $t^{\pm 1}$, and $w_F$ and  $w_E$ are the reduced words on $a_1, \ldots, a_m$ and on $s, t$, respectively, such that $w = w_F w_E$ in $G$.

\section{The distortion of $\langle \lambda \rangle$ in $\Lambda$} \label{s: distortion}

Here we prove:

\begin{prop} \label{prop: dist exp}
For the $H$, $G$, and $\Lambda$ of Section~\ref{intro}, the central subgroup $\langle \lambda \rangle$ of $\Lambda$ is infinite-cyclic and  $\Dist_{\langle \lambda \rangle}^{\Lambda} \mleft(n\mright) \simeq 2^n$. 
\end{prop}

We noted in the introduction that one can see that $\<\lambda\>$ is infinite by observing that $\Lambda$
decomposes as $(H\times\<\lambda\>)\rtimes \<s,t\>$.  The upper bound  on the distortion of $\<\lambda\>$ that we require
is a special case of a well-known general fact: if $1\to Z\to\widetilde{\G}\to\G\to 1$ is a central extension of a finitely
presented group $\G$, and $Z$ is finitely generated, then the distortion of $Z$ in $\widetilde{\G}$ is bounded above by the
Dehn function of $\G$.  In our case,   a standard diagrammatic argument  shows that the Dehn function of $G$  is exponential.
But we give a more detailed argument (Lemma \ref{lem: dist}) because features of its proof give a sharper estimate
that will be needed in  Section~\ref{s: CL in lambda}. It also provides a useful introduction to the use of the notation 
established in the previous section.

\begin{lemma}\label{lem: dist}
There exists a constant $C>1$ such that if $\overline{W}$ is a word in the generators of $\Lambda$ and $W$ is the word obtained from $\overline{W}$ by deleting all letters $\lambda^{\pm 1}$, then $\overline{W} = W_FW_E\lambda^N$ in $\Lambda$ for some $N \in \Z$ and  
\begin{align}
|W_F| & \leq C^{|W|}, \label{eq dist WF} \\
|W_E| & \leq |W|, \ \text{and} \label{eq dist WE}  \\
|N| & \leq |\ell| + \alpha C^{\Sigma\left(W\right)}, \label{eq dist N}
\end{align}
where $\ell$ is the exponent sum of the letters $\lambda^{\pm 1}$ in $\overline{W}$, and $\alpha$ is the number of letters $a_1^{\pm 1}, \ldots, a_m^{\pm 1}$ in $\overline{W}$, and $\Sigma\left(W\right)$ is the maximum of $\left|\sigma\left(W_0\right)\right|$ among all prefixes $W_0$ of $W$.    	
\end{lemma}

\begin{proof} Given how $G = F  \rtimes E$, the words $W = W_F W_E$  in $G$ and $W$ can converted to $W_F W_E$ through a sequence of words all of which equal $W$ in $G$ by shuffling all its letters $s^{\pm 1}$ and $t^{\pm 1}$ to the righthand end---that is,  by making successive substitutions of subwords (i) $s^{\varepsilon}  a^{\mu}_i \mapsto \phi^{-\varepsilon}(a^{\mu}_i) s^{\varepsilon}$  and (ii) $t^{\varepsilon}  a^{\mu}_i \mapsto a^{\mu}_i  t^{\varepsilon}$ for $\varepsilon, \mu  \in \set{\pm 1}$---and then freely-reducing.    

The same sequence of substitutions, except with each type-(ii) substitution producing a central $\lambda^{\varepsilon \mu}$, converts $\overline{W}$ to   $W_F W_E \lambda^N$ in $\Lambda$ through a sequence of words all of which represent the same element of $\Lambda$.  
 
 At an intermediate point in this shuffling process any one of the $(|W| - |W_E|)$ letters $a_i^{\pm 1}$  in $W$ has been replaced by a subword freely equal to $\phi^k(a_i^{\pm 1})$ for some $k$ such that $|k| \leq  \Sigma(W)$---such a subword has length at most $C^{\Sigma(W)}$ for a suitable constant $C >0$.  Accordingly, we can bound the exponent-sum of the $\lambda$ that arise as the $t^{\pm 1}$ in $W$ shuffle past the $a^{\pm 1}_i$ so as to deduce the claimed bound. 
\end{proof}

\begin{proof}[Proof of Proposition~\ref{prop: dist exp}] We have already explained why $\<\lambda\>$ is infinite.

If $\phi$ is any atoroidal automorphism of a free group $F$,  then $w\mapsto |\phi^k(w)|$ grows exponentially for all 
non-trivial $w\in F$, but this is not enough to force exponential distortion in the centre of $\Lambda$.  However, we
have assumed that $\phi$ has homological stretch (which is automatic when  $\phi$ is a positive automorphism),
so  for some $i\in\{1,\dots,m\}$, writing $f(n)$ for
 the exponent sum of the letters in $\phi^n\left(a_i\right)$, we have  $|f(n)|  \simeq 2^n$.
 And  $[t, s^{-n} a_i s^n]$ is a word of length $4n+4$ that equals $\lambda^{f(n)}$ in $\Lambda$,
so $\Dist_{\langle \lambda \rangle}^{\Lambda} \mleft(n\mright) \succeq 2^n$.

The reverse bound follows from Lemma~\ref{lem: dist} because  $\ell=0$, $\alpha \leq |W|$  and   $\Sigma\left(W\right) \leq |W|$ for all words $W$ on the generators of $G$.    
\end{proof}

\section{Conjugators in $H$ and $E$} \label{s: conjugators in H and E}

Accounts of the following proposition  can be found in \cite{Lysenok},  \cite{BrH} pp.\ 451--454,  and \cite{BrRiSa}.    
The proof is a careful coarsening of an observation concerning conjugacy in finitely generated free groups:
there,  working with a free basis,
two words represent conjugate elements if and only if they have  cyclic permutations that are freely equal,
which shows that the constant $C$ of the following proposition can be taken to be $1$ in that setting.

\begin{prop}[Lysenok] \label{Lysenok}
If $\Gamma$ is a hyperbolic group, then (with respect to any fixed finite generating set for $\Gamma$) there exists $C>0$ such that  the conjugator length function satisfies $\CL_{\Gamma}\mleft(n\mright) \leq C n$ for all $n \in \N$. 
\end{prop}

So, given that $H$ is hyperbolic and $E$ is free, and given the length bounds \eqref{eq length under quotients}  and \eqref{eq length under retracts}, we deduce:

\begin{cor} \label{cor:hyp H and E conjugators}
There is a constant $C>0$ such that if $u \sim v$ in $G$ and $n := |u| + |v|$, then there exist  $x_0 \in E$ and $w_0 \in H$  such that $u_E x_0 = x_0 v_E$ in $E$ and  $u_H w_0 =w_0 v_H$ in $H$ (and so in $G$), and
\begin{align}
|x_0|_E & \leq n, \text{ and}   \label{eq x_0}  \\ 
|w_0|_H =|w_0|_G  & \leq  C n.    \label{eq: w0G upper bound} 
\end{align}
\end{cor}

We work with a  torsion-free hyperbolic group $H$ because the following well-known result -- a consequence of Corollary~3.10 of \cite[Chapter III.$\Gamma$]{BrH}, for example -- tells us that centralisers of elements in $H$ have a straightforward structure. It will allow us to   understand centralisers of elements in $G$ in Lemma~\ref{l: Centralizers in hyp case}, which will then help us   understand conjugacy.

\begin{prop}[Gromov]  \label{prop:centralizers in hyp}
If $u$ is a non-identity element of a torsion-free hyperbolic group $\Gamma$, then the centraliser $C_{\Gamma}(u)$ of $u$ is cyclic and is generated by any maximal root of $u$.     
\end{prop}

We will also call on the following fact, which is of surprisingly recent origin \cite[Proposition~4.7]{Bridson13}.

\begin{prop}[\cite{Bridson13}] \label{lengths of roots}  If  $\Gamma$ is a torsion-free hyperbolic group, then it enjoys the  \emph{uniformly monotone cyclics} property: there exists $C >0$ such that for all $h \in \Gamma$ and all non-zero integers $r$ we have $|h^i|_{\Gamma} \leq C |h^r|_{\Gamma}$ for all $0 < i <|r|$. \end{prop}

In the context of $G$, Propositions~\ref{prop:centralizers in hyp} and \ref{lengths of roots} contribute to our next lemma.

  Suppose $u$ is a word  on the generators of $G$.  When $u_E$, $u_F$, and $u_H$ are non-identity elements,  Proposition~\ref{prop:centralizers in hyp} applies to each and tells us that their centralisers in $E$, $F$, and $H$ (respectively) are generated by maximal roots  $u'_E$ of $u_E$, $u'_F$ of $u_F$, and $u'_H$ of $u_H$ (respectively).  Here, take $u'_E$ to be a reduced word in $\set{s,t}^{\pm 1}$, and $u'_F$  a reduced word in $\set{a_1, \ldots, a_m}^{\pm 1}$, and $u'_H$ to be a geodesic word in the generators for $H$. The following lemma records length bounds and in the cases of   $u_F$ and $u'_F$ gives contrasting bounds on word length $| \cdot |$ and on  the lengths $| \cdot |_G$ of shortest representatives in $G$.    

\begin{lemma} \label{lem: length estimates for roots}
There exists a constant $C > 1$ such that if     $u$ is a word of length at most $n$  on the generators of $G$, then (when defined as explained above)   
\begin{align}  
\max \{ |u_E|,  |u_H|   \} &  \leq n,  \label{eq lengths of u_E and u_H}  \\  
\max \{  |u'_E|,  |u'_H|  \} &  \leq C n,  \label{eq lengths of roots as words1}  \\  
\max \{ |u_F|,     |u'_F|   \} & \leq C^n,   \label{eq lengths of roots as words2} \\ 
 |u_F|_G     & \leq 2 n. \label{eq lengths of roots in G} 
\end{align}
\end{lemma}

\begin{proof}
That \eqref{eq lengths of u_E and u_H} holds is immediate, and then       \eqref{eq lengths of roots as words1} follows from Proposition~\ref{lengths of roots}. For \eqref{eq lengths of roots as words2},    Lemma~\ref{lem: dist}\eqref{eq dist WF} and then Proposition~\ref{lengths of roots} give that for a suitable constant $C$, the bounds   $|u_F|   \leq C^n$  and then $|u'_F|   \leq C^n$ hold.       
The triangle inequality applied to $u_F = u u_E^{-1}$  gives \eqref{eq lengths of roots in G}: $|u_F|_G \leq |u|_G + |u_E|_G \leq |u|  + |u_E|  \leq 2n$.
\end{proof}

\section{Centralisers and conjugators in $G$} \label{s: centraliser and conjugators in G}

Our next lemma classifies the centralisers of elements in $G$.   Its statement (ditto Lemma~\ref{l: Centralizers last case} that follows) implicitly assumes that the centralisers of $u_E \neq 1$ in $E$,  $u_F \neq 1$ in $F$, and $u_H \neq 1$ in $H$ are cyclic, which is true by  Proposition~\ref{prop:centralizers in hyp} since these three groups are hyperbolic and torsion-free.  

\begin{lemma}\label{l: Centralizers in hyp case} Suppose $u \in G$. Assuming $u_E \neq 1$, $u_F \neq 1$, and $u_H \neq 1$ (respectively), let $u_E'$, $u_F'$ and $u_H'$ be generators for the centralisers of $u_E$ in $E$,  $u_F$ in $F$, and $u_H$ in $H$ (respectively).  Then --  

\renewcommand{\theenumi}{\alph{enumi}}
\renewcommand{\labelenumi}{\textup{(\alph{enumi}).}}

\begin{enumerate} 
\item  $Z_G(u) = G$ if $u=1$. \label{l hyp:0}
\item  $Z_G(u) = F \times \<u_E'\> $  if $u_H=1$, but $u \neq 1$.   \label{l hyp:a}
\item  $Z_G(u) =  \<u'_F\> \times E_0$ if $u_E =1$, but $u_H\neq 1$.   \label{l hyp:b}
\item $Z_G(u) =  \<u'_F\> \times \<u_E'\>$ if $\sigma(u) =0$, but $u_E \neq 1$ and $u_H\neq 1$. \label{l hyp:b2}
\item  $Z_G(u)$ is cyclic  if  $u_E \neq 1$, $u_H\neq 1$ and $\sigma(u) \neq 0$---indeed, in this case $\sigma(u'_H) \neq 0$  and  $\sigma(u'_E)  \neq 0$, and  if $p$ and $q$ are the integers such that 
 \begin{equation} \label{eq: lcm} 
p \sigma(u'_H) = q \sigma(u'_E) = \textup{lcm}\left(   \sigma(u'_H), \sigma(u'_E) \, \right)
\end{equation}
 and $f \in F$ is such that  $(u'_H)^p = f s^{p\sigma(u'_H)}$ in $H$, then  
 $z = f (u'_E)^q$ generates $Z_G(u)$.
  \label{l hyp:c}  
\end{enumerate} 

\end{lemma}

 \begin{proof}
That Case~\eqref{l hyp:a} makes reference to $u'_E$ presumes that $u_E \neq 1$, which is so because it is implied by $u_H =1$ and $u \neq 1$.  Similarly, Case~\eqref{l hyp:b} can make reference to $u'_F$ because its hypotheses $u_E =1$ and $u_H \neq 1$  imply that $u_F \neq 1$.  And \eqref{l hyp:b2} can make reference to $u'_F$ because $u_H = u_F s^{\sigma (u)} \neq 1$ and $\sigma (u)=0$ imply $u_F \neq 1$.   

We will prove the five cases in turn.

  	 {\eqref{l hyp:0}} is immediate. 
  	
 	\eqref{l hyp:a}  Suppose $u_H =1$.  If $w = w_F w_E$ has  $w_E \in \<u_E'\>$, then $w_E u_E = u_E w_E$.  Further,   $\sigma(u) = 0$, and so  $\sigma(u'_E) = 0$ and $u'_E$ commutes with all elements of $F$.  So  $F \times \<u_E'\> \leq Z_G(u)$, the direct product being appropriate here because $u_E' \in E_0$.  And the reverse inclusion holds because if $w \in  Z_G(u)$, then $w_E u_E = u_E w_E$ and so $w_E \in \<u_E'\>$.   
 	
 	\eqref{l hyp:b} Suppose $u_E =1$ but $u_H \neq 1$.  Then $u = u_F \neq 1$ and so $\< u'_F \> \times E_0 \leq Z_G(u)$.  And if $w \in Z_G(u)$, then $w_H  \in Z_G(u)$ and so, because $H$ is hyperbolic, $w_H$ is a power of a maximal root $u'$ of $u$ in $H$.  But $\sigma(u)=0$, so $\sigma(u')=0$,  so  $\sigma(w_H)=0$, and so $\sigma(w)=0$.  So $w \in F \times E_0$, and   $w_F \in Z_F(u_F) =  \<u'_F\>$ because all elements of $E_0$ commute with all elements of $F$.    
 	
 	 	\eqref{l hyp:b2}    	 	A $w\in G$ commutes with $u$ if and only if $u_F u_E \ w_F w_E  = w_F w_E \ u_F u_E$  in   $G$.   But $\sigma(u) =0$ implies that $u_E \in E_0$ and so $u_E$ commutes with all elements of $F$, as does all roots of $u_E$ in $E$.  So  $w u=u w$ in $G$ if and only if $u_H w_H  =  w_H u_H$  in  $H$    and $u_E w_E  =  w_E u_E$   in $E$.  That  $Z_G(u) =  \<u'_F\> \times \<h_E'\>$ follows.

 	\eqref{l hyp:c}   We will argue that this follows from the special case of our next lemma where $v=u$ and $x_0 = w_0 =1$.  
 	
 	Let $p$ and $q$ be as per \eqref{eq: lcm}.  The lemma tells us that if $z \in G$ satisfies $z_H = (u'_H)^p$ in $H$ and $z_E =  (u'_E)^q$, then  $z \in Z_H(u)$. These conditions define $z$ because they necessitate that if  	 $f \in F$ is such that $z_H = f s^{\sigma(z_H)}$ then	 $z = f z_E$,  and \eqref{eq: lcm} implies that this $z$ satisfies   $z_H = (u'_H)^p$.  
 	
 		Moreover the lemma tells us that \emph{any}  $z' \in Z_H(u)$ has $z'_H = (u'_H)^{pr}$ in $H$ and $z'_E =  (u'_E)^{qr}$ for some $r \in \Z$. But
 		$(z^r)_H =(z_H)^r = ((u'_H)^p)^r = (u'_H)^{pr} = z'_H$ and likewise $(z^r)_E = z'_E$, and so $z' = z^r$.   
 	 \end{proof}

\begin{lemma}\label{l: Centralizers last case}   Suppose $u, v \in G$ and that  $u_E \neq 1$, $u_H\neq 1$ and $\sigma(u) \neq 0$.  
Let $u_E'$ and $u_H'$ be generators for the centralisers of $u_E$ in $E$ and $u_H$ in $H$ (respectively).  Suppose $x_0 \in E$ and $w_0 \in H$ are such that $u_E x_0 = x_0 v_E$ in $E$ and $u_H w_0 = w_0 v_H$ in $H$.   

Then $\sigma(u'_H)$ and   $\sigma(u'_E)$ are non-zero.  Further, $g \in G$ satisfies  $ug=gv$ in $G$ if and only if $g$ is represented by  a word  $w = w_H s^{-\sigma(w)} w_E$ such that $w_H = (u'_H)^p w_0$ in $H$ and $w_E =  (u'_E)^q x_0$ in $E$ and $p$ and $q$ are integers satisfying  
\begin{equation} \label{diophantine}
 p \, \sigma(u'_H) - q \, \sigma(u'_E) = \sigma(x_0) - \sigma(w_0). 
 \end{equation} 
\end{lemma}

\begin{proof} The reason that $\sigma(u'_H)$ and   $\sigma(u'_E)$ are non-zero is that $\sigma(u_H) \neq 0$  is a multiple of each.

We claim that a word $w \in G$ satisfies $uw = wv$ in $G$ if and only if 
\begin{align} 
 u_H w_H & =  w_H v_H   \text{ in } H \text{ and }    \label{eq in H} \\ 
   u_E w_E & =  w_E v_E   \text{ in } E.  \label{eq in E}
\end{align}
The `only if' direction is apparent from equating the images of $uw$ and $wv$ under the map $G \onto H$ and then under $G \onto E$.  
For the `if' direction, we calculate that 
\begin{align} 
 u w & =  u_H s^{- \sigma(u)} u_E    w_H s^{- \sigma(w)} w_E        \label{eq uv=vw1} \\ 
     & =  u_H     w_H s^{- \sigma(w)} s^{- \sigma(u)} u_E w_E      \label{eq uv=vw2} \\
     & =  w_H     v_H s^{- \sigma(v)} s^{- \sigma(w)} w_E v_E      \label{eq uv=vw3} \\
     & =  w_H      s^{- \sigma(w)} w_E v_H s^{- \sigma(v)} v_E      \label{eq uv=vw4} \\
     & =  w v \label{eq uv=vw5}
\end{align}
because \eqref{eq uv=vw2} $w_H s^{- \sigma(w)} \in F$ and  $s^{- \sigma(u)} u_E \in E_0$ and so they commute,    \eqref{eq uv=vw3} by the hypotheses \eqref{eq in H} and \eqref{eq in E} hold and also by $\sigma(u) = \sigma(v)$, which follows from either, and \eqref{eq uv=vw4} $v_H s^{- \sigma(v)} \in F$ and  $s^{- \sigma(w)} w_E \in E_0$ and so they commute.

The result then follows because any $w$ satisfying these equivalent conditions must equal
$w = w_F w_E =  w_H s^{-\sigma(w_H)} w_E$ in $G$, and have        
$w_H = (u'_H)^p w_0$ in $H$ for some $p \in \Z$ and $w_E = (u'_E)^q x_0$ in $E$ for some $q \in \Z$. Further,  \eqref{diophantine} must hold because $\sigma(w_H) = \sigma(w_E)$.  
\end{proof}

The following lemma shows that $\CL_G(u,v)$ is at most a constant times  $(|u| + |v|)^2$ under the hypotheses of Case~\eqref{l hyp:c} of Lemma~\ref{l: Centralizers in hyp case}.  In fact, $\CL_G(u,v)$ is at most a constant times $|u| + |v|$  in Cases~\eqref{l hyp:0}--\eqref{l hyp:b2}, and so $\CL_G\mleft(n\mright) \preceq n^2$.  Most of the estimates establishing this can be found within our arguments in Section~\ref{s: CL in lambda}, but it is not a result we will need, so we do not include an explicit proof.

Like in Lemma~\ref{lem: dist},  $\Sigma\left(w\right)$ denotes the maximum of $\left|\sigma\left(w_0\right)\right|$ among all prefixes $w_0$ of $w$.   

\begin{lemma} \label{lem:best w}
There exists a constant $C_0>0$ such that  for all words $u$ and $v$ as per Lemma~\ref{l: Centralizers last case},  there exists a word $w$ such that $uw =wv$ in $G$ and 
\begin{align}
|w|_G & \leq C_0 n^2,  \label{eq wG quadratic} \\ 
\Sigma(w) & \leq  C_0 n,  	\label{eq Sigma linear} \\
|\sigma(w)| & \leq  C_0 n.   	\label{eq sigmaw linear}  
\end{align}
\end{lemma}

 \begin{proof} 
 
Let $C$ be the maximum of the constants of   Corollary~\ref{cor:hyp H and E conjugators} and Lemma~\ref{lem: length estimates for roots}.  Lemma~\ref{lem: length estimates for roots} applies to $u$ and, given that $|u'_E|_E = |u'_E|_G$ and $|u'_H|_H = |u'_H|_G$, tells us that
\begin{equation} \label{eq lengths of roots of uE and uH in G} 
\max \{ |u'_E|_E,   |u'_H|_H  \} \leq C  n. 
\end{equation}
 Per Corollary~\ref{cor:hyp H and E conjugators} let $x_0 \in E$ and $w_0 \in H$ be such that $u_E x_0 = x_0 v_E$ in $E$ and $u_H w_0 = w_0 v_H$ in $H$ and 
 \begin{equation} \label{eq x0 w0 bound} 
 |x_0|  \leq    n \ \text{ and }  \  |w_0| \leq  C n.   
 \end{equation}   
 
 If $w$ is \emph{any} conjugator as per Lemma~\ref{l: Centralizers last case}, then  
\begin{equation}  \label{eq w}
w = w_H s^{-\sigma(w)} w_E = (u'_H)^p w_0 s^{-\sigma(w)}  (u'_E)^q x_0  
\end{equation} 
so that  $w_H = (u'_H)^p w_0$ in $H$ and $w_E =  (u'_E)^q x_0$ in $E$, and  $p$ and $q$ are integers  
satisfying \eqref{diophantine}. 
And then, by the triangle inequality and then \eqref{eq lengths of roots of uE and uH in G} and \eqref{eq x0 w0 bound}, 
 \begin{align} 
|w|_G  & \leq |p| \, |u'_H|_H + |w_0|  + |\sigma(w)|  + |q| \, |u'_E|_E + |x_0| \label{eq gG bound2}  \\ 
       &\leq  C_1   n \, \max \set{|p|, |q|}   \label{eq gG bound3}   
\end{align}
for a suitable constant $C_1 > 0$.

 Now,  $u \sim v$ and so there exists such a conjugator $w$ and therefore there exist integers $p$ and $q$ satisfying \eqref{diophantine}.
 But then, because  $\sigma(u'_H)$ and   $\sigma(u'_E)$ are non-zero,  Corollary~\ref{cor Bezout lcm} applies  so as to tell us that there exist integers $p$ and $q$ satisfying  \eqref{diophantine}  such that 
\begin{align}
|p \sigma(u'_H)|  & \leq     \lcm \set{|\sigma(u'_H)|, |\sigma(u'_E)|} \ \text{ and}  \label{eq bound on p and q1} \\  
|q \sigma(u'_E)| & \leq    \max \set{ \lcm \set{|\sigma(u'_H)|, |\sigma(u'_E)|}, |\sigma(x_0) - \sigma(w_0)|}.  \label{eq bound on p and q2}
\end{align}

Both $\sigma(u'_H)$ and  $\sigma(u'_E)$ divide $\sigma(u)$, and so $ \lcm \set{|\sigma(u'_H)|, |\sigma(u'_E)|} \leq |\sigma(u)| \leq n$.    
And $|\sigma(x_0) - \sigma(w_0)| \leq|\sigma(x_0)| + |\sigma(w_0)| \leq |x_0|_G + |w_0|_G \leq (1 + C) n$, the last by \eqref{eq x0 w0 bound}.  So 
 \begin{equation} \label{eq bound on p and q by n}
\max \set{|p \sigma(u'_H)|, |q \sigma(u'_E)|}  \leq    (1 +C)n.
\end{equation}

  Because $\sigma(u'_H)$ and $\sigma(u'_E)$ are non-zero integers,  \eqref{eq bound on p and q by n} implies that 
  \begin{equation} \label{eq p and q}
  \max \set{|p|, |q|}  \leq (1+ C)n.
  \end{equation}
    And, for a suitable constant $C_0$,  the   conjugator $w$ associated to this $p$ and $q$ per \eqref{eq w} witnesses to the bound \eqref{eq wG quadratic} because of \eqref{eq gG bound3}  and  \eqref{eq p and q}.

    Finally, we will explain why this $w$  witnesses to the three bounds claimed in the lemma.  The bound \eqref{eq wG quadratic} holds because of   \eqref{eq gG bound3} and \eqref{eq p and q}.         For  \eqref{eq Sigma linear}, first  observe that any prefix $\pi$ of  $(u'_H)^p$ has $|\sigma(\pi)|$ at most a constant times $n$ because of \eqref{eq bound on p and q by n} and because  $|u'_H|$ is at most a constant times $n$  by Lemma~\ref{lem: length estimates for roots}. The same holds for any prefix of $(u'_E)^q$, likewise.
   
   So  \eqref{eq Sigma linear} holds because the remainder of $w = (u'_H)^p w_0 s^{-\sigma(w)}  (u'_E)^q x_0$---namely, $w_0$, $s^{-\sigma(w)}$, and $x_0$---has total length at most a constant times $n$.    
    
Finally, the bound \eqref{eq sigmaw linear} follows immediately from \eqref{eq Sigma linear}.
\end{proof}

\section{Why $\CL_{\Lambda}(n) \preceq 2^n$}  \label{s: CL in lambda}

We again recall the context: $\Lambda$ is the central extension of $G$ constructed as 
\begin{align*}
H & = \langle a_1, \ldots, a_m, s \mid s^{-1} a_i s = \phi(a_i)  \  \   \forall i \  \rangle \\ 
G &  =     \langle \, H, t  \mid [t,a_1]=  \cdots  = [t,a_m]=1   \rangle \\ 
\Lambda &  =     \langle \, H, t, \lambda  \mid [t,a_1]=  \cdots  = [t,a_m]=\lambda, \  \lambda \textup{ central }   \rangle.  
\end{align*}

A word in $a_1, \ldots, a_m, s, t$ can represent  an element $g \in G$ or $\o{g} \in \Lambda$, with $\o{g}$ being a particular lift of $g$.  Both contexts will arise in our proof and we will take care to distinguish them.  

A number of constants $C_i$ will arise in the following proof.  All will be independent of $n$, but they may depend on $\Lambda$, $m$, and our choices of generating sets.    

Suppose $\o{u}$ and $\o{v}$ are words on the generators of $\Lambda$ such that  $\o{u}  \sim \o{v}$ in $\Lambda$.  Let $n = |\o{u}|  + |\o{v}|$.    

Let $u$ and $v$ be the words obtained from $\o{u}$ and $\o{v}$, respectively, on deleting all letters $\lambda^{\pm 1}$.  Then $u \sim v$ in $G$ and $|u| + |v| \leq n$.  If  we write $L_u$ and $L_v$ for the exponent-sums of the letters $\lambda^{\pm 1}$ in $\o{u}$ and $\o{v}$, respectively, then  $\o{u} = u \lambda^{L_u}$ and $\o{v} = u \lambda^{L_v}$ in $\Lambda$ and
\begin{equation} \label{eq: Lu Lv upper bound} 
 |L_u| + |L_v|   \leq   n.   
  \end{equation}

Lemma~\ref{lem: length estimates for roots} applies to both $u$ and $v$ (when $u_E$, $u_F$, $u_H$, $v_E$, $v_F$, and $v_H$ are non-identity elements).  We will call on the estimates it gives in what follows.

Lemma~\ref{l: Centralizers in hyp case} sets out the structure of $Z_G(u)$ in five mutually exclusive cases.  We will examine those cases in turn and find suitable upper bounds for $\CL_{\Lambda}(\o{u},\o{v})$.

\emph{Case \eqref{l hyp:0}. $u =1$ in $G$, and so $Z_G(u) = G$.}  In this case $v=1$ in $G$ also and $\o{u}$ and $\o{v}$ are central in $\Lambda$, and so $\CL_{\Lambda}(\o{u},\o{v}) =0$.

\emph{Case \eqref{l hyp:a}. $u_H =1$ in $H$ and $u \neq 1$ in $G$, and so   $Z_G(u) = F \times \<u_E'\>$.}        
In this case $u=u_E$ and $v=v_E$ in $G$  and $\sigma(u) = \sigma(v)=0$.  And $|u_E| \leq |u|$ and $|v_E| \leq |v|$, and so $|u_E| + |v_E| \leq n$.   So, given that $\Dist_{\langle \lambda \rangle}^{\Lambda} \mleft(n\mright) \simeq 2^n$ (by Proposition~\ref{prop: dist exp}), for a suitable constant $C_1>0$, we have that in $\Lambda$,  
\begin{equation} \label{eq L'uL'v} 
u = u_E \lambda^{L'_u} \text{ and } v = v_E \lambda^{L'_v}    \text{ for some }  L'_u, L'_v \in \Z \text{ with } |L'_u|, |L'_v| \leq C_1^n.
\end{equation}         

 Now, $u_E \sim v_E$ in $E$, and so by Corollary~\ref{cor:hyp H and E conjugators} there  exists $x_0 \in E$ such that $u_E x_0 = x_0 v_E$ in $E$ and $|x_0|_E  \leq n$.  
Because  $Z_G(u)  = \set{ f \, (u'_E)^q \mid f \in F, q \in \Z }$,  the words $w$ such that $uw = wv$ in $G$ are those expressible as $w = f \, (u'_E)^q x_0$ for some $f \in F$ and $q \in \Z$.  For such a $w$, we have $u w = wv$ in $G$ and so $u w = wv \lambda^N$  in $\Lambda$ for some $N \in \Z$. 
Indeed, in $\Lambda$    
\begin{align} 
  u w  & =  u_E   f   \left(u'_E\right)^q x_0 \ \lambda^{L'_u} \label{eqb1} \\ 
 & =  f     \left(u'_E\right)^q   u_E  x_0 \ \lambda^{L'_u+L} \label{eqb2} \\ 
  & =  f     \left(u'_E\right)^q x_0  v_E   \ \lambda^{L'_u+L} \label{eqb3} \\ 
 & =  w v \lambda^{L'_u-L'_v+L},     \label{eqb4}
\end{align}
where for  \eqref{eqb1} and  \eqref{eqb4} we use  \eqref{eq L'uL'v}, for \eqref{eqb2} we use that    in $\Lambda$ 
\begin{equation} \label{eq fLambda} 
     u_E    f =  f u_E \lambda^L
  \end{equation}
  for some $L \in \Z$ and $u_E$ commutes with $u'_E$, and for \eqref{eqb3} we use that $u_e x_0 =x_0 v_E$ in $E$ and so in $\Lambda$.   

So, because   $\o{u} = u \lambda^{L_u}$ and $\o{v} = v \lambda^{L_v}$ in $\Lambda$,  we have  $   \o{u}  w = w \o{v}$ in $\Lambda$  if and only if  $L = L_v - L_u + L'_v - L'_u$.  Because $\o{u} \sim \o{v}$ in $\Lambda$, there exists $f \in F$ such that \eqref{eq fLambda} holds for this value of $L$.

For $i =1, \ldots, m$, we have $u_E   a_i = a_i u_E$ in $G$ because $\sigma(u_E) = \sigma(u)=0$.  
So  $u_E   a_i = a_i u_E \lambda^{n_i}$ in $\Lambda$ for a unique integer $n_i$. Let $\alpha_i$ be the exponent-sum of the letters $a_i$ in $f$.
Then from  \eqref{eq fLambda} we get that  
\begin{equation} \label{eq Bezout ni} 
L = \alpha_1 n_1 +  \cdots + \alpha_m n_m.
  \end{equation}
Because $|u_E| \leq C n$ by  \eqref{eq lengths of roots as words1}, Proposition~\ref{prop: dist exp} gives that for a suitable constant $C_2>1$, for all $i$,
\begin{equation} \label{eq  ni exp bound} 
|n_i| \leq C_2^n.  
  \end{equation}
  Now,   $|L| \leq C_3^n$  for a suitable constant $C_3 >1$ by  \eqref{eq: Lu Lv upper bound} and  \eqref{eq L'uL'v}.  Given our quantification of B\'ezout of Corollary~\ref{cor Bezout multi}, these estimates tell us that there exist $\alpha_1, \ldots, \alpha_m$ satisfying 
 \eqref{eq Bezout ni} with $|\alpha_i| \leq \max\{C_2^n, C_3^n\}$ for all $i$.  But then $f = a_1^{\alpha_1} \cdots a_m^{\alpha_m}$    satisfies \eqref{eq fLambda} and,    calling on  \eqref{eq x_0}, we get that  $w = f x_0$ is a word of length at most $C_4^n$ for a suitable constant $C_4>1$   such that  $\o{u}  w = w  \o{v}$ in $\Lambda$, as required.  (In the above analysis the value of $q$ in $w$ played no role, so we can choose it to be zero.)

\emph{Case \eqref{l hyp:b}. $u_E =1$ and $u_H  \neq 1$, and so $Z_G(u) =  \<u'_F\> \times E_0$.}        
In this case $u=u_F$ and $v=v_F$ in $G$ and $\sigma(u) = \sigma(v)=0$.  
By \eqref{eq lengths of roots as words2}, $\max \set{ |u_F|, |u'_F|,  |v_F|} \leq C_1^{n}$.   
Further, because  $|u|_G +  |v|_G  \leq n$ and, by \eqref{eq lengths of roots in G},  $\max \set{|u_F|_G, |v_F|_G}  \leq C_0 n$, and so there exists a constant $C_5 >1$ such that   $u = u_F \lambda^{L'_u}$ and  $v = v_F \lambda^{L'_v}$ in $\Lambda$ for some integers
\begin{equation} \label{eq: L'_u upper bound} 
 |L'_u| + |L'_v|   \leq  C_5^{n}.   
  \end{equation}

Now, $u \sim v$ in $H$ and so by Corollary~\ref{cor:hyp H and E conjugators} there exists a $w_0 \in H$ such that  $u w_0 =w_0 v$ in $H$ (and so in $G$) and, for a suitable constant $C_6>1$,
\begin{equation}  
|w_0|_G = |w_0|_H \leq C_6\left(|u|_H + |v|_H\right) \leq  C_6 n.   
  \end{equation}

Because $Z_G\left(u\right) =  \langle u'_F \rangle \times E_0$, the elements of $G$  that conjugate $u$ to $v$ in $G$ are those represented by the words $w =\left(u'_F\right)^p   x   w_0$ 
for $p \in \Z$ and $x \in E_0$.  We calculate in $\Lambda$ that for such $w$,      
\begin{align}
 u w  & =                 u_F      \left(u'_F\right)^p x w_0 \ \lambda^{L'_u } \label{eq: 3a} \\ 
 & =          \left(u'_F\right)^p  x u_F      w_0  \lambda^{L+L'_u} \label{eq: 3b} \\ 
  & =           \left(u'_F\right)^p  x       w_0 v_F  \lambda^{L+M+L'_u} \label{eq: 3c} \\ 
 & =                 w v    \lambda^{L+M+ L'_u- L'_v } \label{eq: 3d}  
\end{align}    
because \eqref{eq: 3a}   $u = u_F \lambda^{L'_u}$, \eqref{eq: 3b}  $u_F$ and $u'_F$ commute and   \begin{equation} \label{eq: L equals1}  u_F   x  = x u_F    \lambda^{L} \end{equation} for some $L \in \Z$,   \eqref{eq: 3c}    $u_F w_0=  w_0   v_F  \lambda^{M}$ for some $M \in \Z$, and \eqref{eq: 3d}   $v = v_F \lambda^{L'_v}$. Then, because   $\o{u} = u \lambda^{L_u}$ and $\o{v} = v \lambda^{L_v}$ in $\Lambda$,  for $\o{u}  w = w \o{v}$ in $\Lambda$, we must have  
 \begin{equation} \label{eq: L equals}
L = L_v - L_u  - M - L'_u + L'_v.  
\end{equation}

Now, $|M| \leq C_7^n$ for a suitable constant $C_7 >1$ by Proposition~\ref{prop: dist exp} given  \eqref{eq: w0G upper bound} and that   $\max \set{|u_F|_G, |v_F|_G}  \leq C_0 n$.  In light of this,    \eqref{eq: Lu Lv upper bound}, and \eqref{eq: L'_u upper bound}, the equation \eqref{eq: L equals} then tells us that for a suitable constant $C_8 >1$, 
 \begin{equation} \label{eq: L bound C8}
|L| \leq C_8^n.  
\end{equation}

Because $E_0$ is generated by $\set{s^{-i} t s^i \mid i \in \Z }$ we get:   

\emph{Observation.} The set $\mathcal{L}$  of integers $L$ such that $x^{-1}   u_F   x  = u_F    \lambda^{L}$ in $\Lambda$ for some $x \in E_0$ is the ideal $(\mathcal{L}) \leq \Z$ generated by the integers $\set{ L_i | i \in \Z}$ defined by  $$(s^{-i} t s^i)^{-1}   u_F   (s^{-i} t s^i)  = u_F    \lambda^{L_i}.$$

\emph{Claim.}  $\mathcal{L}$ equals the ideal $\mathcal{L}' = ( L_0, \ldots, L_{m-1} )$.  For $i=1, \ldots,  m$, let $\alpha_i$ be the exponent-sum of the letters $a_i$ in $u_F$.  Let $\Phi \in \textup{GL}_m(\Z)$ be the $m \times m$ matrix of the map (with respect to the basis $a_1$, \ldots, $a_m$) induced by $\phi: F \to F$ on abelianising $F$.  Then \begin{equation} \label{eq the exp sum calculation} L_i =  \begin{bmatrix}   1 & \cdots & 1       \end{bmatrix} \ \Phi^i  \begin{bmatrix}
           \alpha_{1} \\
           \vdots \\
           \alpha_{m}
         \end{bmatrix}. \end{equation}  
By the Cayley--Hamilton Theorem, $\Phi^m$ is a $\Z$-linear combination of  $\Phi^{m-1}, \ldots, \Phi, I$ and so  $L_i \in \mathcal{L}'$ for all $i \geq m$. Indeed, the constant term of the characteristic polynomial of $\Phi$  is $\textup{det}(\Phi) = \pm 1$.  So $\Phi$ can be expressed as a $\Z$-linear combination of  $\Phi^2, \ldots, \Phi^{m+1}$.     So  $L_i \in \mathcal{L}'$ for all $i <0$. Therefore   $\mathcal{L} = \mathcal{L}'$ as claimed.

Because $\o{u} \sim \o{v}$ in $\Lambda$, the $L$ of   \eqref{eq: L equals} is in $\mathcal{L}$ and there exists $x \in F$ such that \eqref{eq: L equals1} holds.

In the light of the above claim, we learn from Corollary~\ref{cor Bezout multi}  that there exist $n_0, \ldots, n_{m-1} \in \Z$ such that 
\begin{equation} \label{eq Lns}
L = n_0 L_0 + \cdots + n_{m-1} L_{m-1}
\end{equation}
 and  for all $i$, 
 \begin{equation} \label{eq nis}
 |n_i| \leq \max \{ |L_0|, \ldots, |L_{m-1}|, |L | \}. 
 \end{equation}
    But for suitable a constant  $C_9 >1$,    $|L_{j}| \leq  C_9^{|j|} |u_F|$ for all $j \in \Z$.  So, for the constant $C_{10}  =  C_9^{|m-1|}$,  $$\max \{ |L_0|, \ldots, |L_{m-1}| \} \leq C_{10} |u_F|$$    
and therefore, in light also of  the bounds on $|u_F|$ from  \eqref{eq lengths of roots as words2} and on $|L|$ from \eqref{eq: L bound C8},  for a suitable constant $C_{11} >1$, we have $|n_i| \leq C_{11}^n$ for $i = 0, \ldots, m-1$.       
     So  $$x =  (s^{-0} t^{n_0} s^0)  \ (s^{-1} t^{n_1} s^1) \  \cdots \ (s^{-(m-1)} t^{n_{m-1}} s^{m-1})  $$ 
is an element of  $E_0$ for which  $x^{-1}   u_F   x  = u_F    \lambda^{L}$ and $|x| \leq C_{12}^n$ for a suitable constant $C_{12} >0$.  And then  $w = x w_0$  satisfies $ \o{u}  w = w \o{v}$ and, given \eqref{eq: w0G upper bound}, we learn that $w$ is sufficiently short.

\emph{Case \eqref{l hyp:b2}. $\sigma(u) =0$,   $u_E \neq 1$, and $u_H  \neq 1$, and so $Z_G(u) =  \<u'_F\> \times \<u_E'\>$.}        

We have $u_H \sim v_H$ in $H$, and so like in Case~\eqref{l hyp:b} there exists  $w_0 \in H$ such that
\begin{equation} \label{eq w0}
  u_H w_0 =w_0 v_H \text{ in }  H \text{ and } |w_0|_G \leq C_6 n. 
  \end{equation}
Also   $u_E \sim v_E$ in $E$, and so like in   Case~\eqref{l hyp:a} there exists $x_0 \in E$ such that
\begin{equation} \label{eq x0}
 u_E x_0 = x_0 v_E \text{ in }  E \text{ and } |x_0| \leq n.   
\end{equation}

There exists $w \in G$ such that  $u w  =w  v$ in $G$, and therefore such that $u_E w_E = w_E v_E$ in $E$ and $u_H w_H =w_H v_H$ in $H$.  Now, $u_E$ is a power of $u'_E$ and  $\sigma(u_E) = \sigma(u) =0$, and so $\sigma(u'_E)=0$.  So $\sigma(x_0) = \sigma(w_E)$ because $w_E^{-1} x_0 \in Z_E(u_E) = \langle u'_E \rangle$.  For a similar reason $\sigma(w_0) = \sigma(w_H)$. 
But then $\sigma(x_0) = \sigma(w_0)$ because $\sigma(w) = \sigma(w_E) = \sigma(w_H)$.   

Next observe that $u_H = u_F$ and $v_H = v_F$ because  $\sigma(u) =0$.  So
$u_F w_H = u_H w_H =w_H v_H = w_H v_F$.  Then $u_F w_H = w_H v_F$ because $w_H = w_F s^{\sigma(w)}$ and $\sigma(x_0) = \sigma(w_0)$.

That $\sigma(x_0) = \sigma(w_0)$ implies that for $f \in F$ we have  $s^{-\sigma(w_0)} f s^{\sigma(w_0)} = x_0^{-1} f x_0$ in $G$. So, given that $w_0 =  (w_0)_F s^{\sigma(w_0)}$ and  $u_H w_0 =w_0 v_H$,    
\begin{equation} \label{eq s}
u_F  (w_0)_F x_0 = (w_0)_F x_0 u_F.
\end{equation}
 in $G$. Let $w_1$ be a minimal-length word that equals $(w_0)_F x_0$ in $G$.  Then in $G$,  
\begin{align}
   u w_1  & = u_F u_E  (w_0)_F x_0 \label{eq1}  \\
   & = u_F  (w_0)_F    u_E x_0  \label{eq2}  \\
& = u_F  (w_0)_F     x_0 v_E    \label{eq3} \\ 
& = (w_0)_F x_0 v_F        v_E    \label{eq4}  \\
& = w_1      v     \label{eq5} 
\end{align}
because \eqref{eq1} and \eqref{eq5} hold by definition,  \eqref{eq2} is a consquence of  $\sigma(u_E) =0$, which implies that $u_E$  and $(w_0)_F$ commute,    \eqref{eq3} is by \eqref{eq x0}, and \eqref{eq4} is by  \eqref{eq s}.

 By definition, $(w_1)_F = (w_0)_F$,  $(w_1)_E = x_0$, and $(w_0)_F = w_0 (w_0)^{-1}_E$.  So $(w_1)_F = w_0 (w_0)^{-1}_E$ and by the triangle inequality, $|(w_0)_F|_G \leq |w_0|_G + |(w_0)_E|_G$, which is at most $2|w_0|_G$.  So  
\begin{equation} \label{eq w1G length}
|w_1|_G \leq  |(w_0)_F|_G  + |x_0|    \leq  2|w_0|_G  + |x_0|  \leq (2C_6+1)n,  
\end{equation}
where the last inequality is by \eqref{eq w0} and \eqref{eq x0}.

Because $Z_G(u) =  \<u'_F\> \times \<u_E'\>$,
the words $w$ such that $uw = wv$ in $G$ are those expressible as $w =(u'_F)^p \, (u'_E)^q   w_1$ 
for some $p,q \in \Z$. 
Now, 
\begin{equation} \label{eq inG}
 u_E   u'_F   =  u'_F   u_E \ \text{ and } \ u_F u'_E =  u'_E u_F  \text{ in } G
\end{equation}
because $\sigma(u_E) = \sigma(u'_E)=0$.  So    $ u_E   u'_F   =  u'_F   u_E \lambda^{P}$  and $u_F u'_E =  u'_E u_F \lambda^{Q}$ for some integers $P$ and $Q$, which we use in \eqref{eq: 4b} and  \eqref{eq: 4e} of the following calculation.  In  $\Lambda$, for such $w$,        
\begin{align}
   u w  & =   u_F u_E      \left(u'_F\right)^p \left(u'_E\right)^q w_1   \lambda^{L'_u} \label{eq: 4a} \\ 
 & =  u_F   \left(u'_F\right)^p  u_E      \left(u'_E\right)^q w_1   \lambda^{L'_u + pP}   \label{eq: 4b}  \\
  & =  \left(u'_F\right)^p    u_F  u_E      \left(u'_E\right)^q w_1   \lambda^{L'_u+ pP}  \label{eq: 4c}   \\
  & =  \left(u'_F\right)^p    u_F   \left(u'_E\right)^q u_E     w_1   \lambda^{L'_u+ pP}  \label{eq: 4d}   \\
  & =  \left(u'_F\right)^p      \left(u'_E\right)^q u_F  u_E     w_1   \lambda^{L'_u+ pP +qQ}  \label{eq: 4e}   \\
  & = w v \lambda^{pP+qQ +N},     \label{eq: 4f} 
\end{align}    
where the remaining steps are explained by: \eqref{eq: 4a}   $u = u_F u_E \lambda^{L'_u}$ in $\Lambda$ for some integer $L'_u$,    \eqref{eq: 4c}    $u_F   \left(u'_F\right)^p$ freely equals  $\left(u'_F\right)^{p} u_F$,  \eqref{eq: 4d}   $u_E     (u'_E)^q$ freely equals $(u'_E)^qu_E$, 
  and  $u = u_F u_E \lambda^{L'_u}$, and \eqref{eq: 4f} $u = u_F u_E$ and   $u  w_1 =  w_1 v \lambda^N$ for some integer $N$.

Similarly to the earlier cases,  because   $\o{u} = u \lambda^{L_u}$ and $\o{v} = v \lambda^{L_v}$ in $\Lambda$, we have $\o{u}  w = w \o{v}$ in $\Lambda$  if and only if  
\begin{equation} \label{eq pq bezout}
L_v - L_u + pP +qQ+N =0,
\end{equation}
and because  $\o{u} \sim \o{v}$ in $\Lambda$,  there exist  $p, q \in \Z$ satisfying this equation.   
Now  we claim that $|P|, |Q|, |N + L_v - L_u| \leq C_{13}^n$ for a suitable constant $C_{13}>1$.  Indeed, Lemma~\ref{lem: dist} applied to $u_E   u'_F$,  which equals $u'_F   u_E \lambda^{P}$ in $\Lambda$, gives that $|P| \leq    |u'_F| C^{\Sigma\left( u_E   u'_F \right)}$. And so $|P|  \leq C_{14}^n$  for a suitable constant $C_{13} >1$, because of Lemma~\ref{lem: length estimates for roots} and because $\Sigma\left( u_E   u'_F \right) \leq |u_E|$.  
Likewise, applied to $u_F u'_E$, which equals $u'_E u_F \lambda^{Q}$ in $\Lambda$, it gives  $|Q| \leq    |u_F| C^{\Sigma\left( u_F u'_E \right)}$.  So $|Q|  \leq C_{13}^n$ because of Lemma~\ref{lem: length estimates for roots} and because $\Sigma\left( u_F u'_E \right) = \Sigma\left(u'_E \right) \leq |u'_E|$. And  Proposition~\ref{prop: dist exp} applied to  $v^{-1} w_1^{-1} u  w_1 =  \lambda^N$ gives that $|N|  \leq C_{13}^n$ because of $|u| + |v| \leq n$ and \eqref{eq w1G length}.  But then by adjusting the constant suitably we can get $|N + L_v - L_u| \leq C_{13}^n$  because $|L_u| + |L_v|   \leq   n$.

So  by the quantification of B\'ezout in Corollary~\ref{cor Bezout multi}, there exist  $p$ and $q$ satisfying \eqref{eq pq bezout} such that  $|p|, |q| \leq C_{13}^n$. And then $\o{u} w =w \o{v}$ in $\Lambda$ and the length of   $w =(u'_F)^p \, (u'_E)^q   w_1$ is $|p| |u'_F| + |q|  |u'_E| +   | w_1|$  is at most  $C_{14}^n$ for a suitable constant $C_{14} >1$, calling on Lemma~\ref{lem: length estimates for roots} for bounds on the remaining terms.

\emph{Case \eqref{l hyp:c}. $\sigma(u) \neq 0$, $u_E \neq 1$, and $u_H  \neq 1$.}  
 
Let $p$ and $q$ be the integers such that  
$p \sigma(u'_H) = q \sigma(u'_E) = \textup{lcm}\left(   \sigma(u'_H), \sigma(u'_E) \, \right)$.
Lemma~\ref{l: Centralizers in hyp case} tells us that $z = f (u'_E)^q$ generates $Z_G(u)$, where  $f$ is the element of $F$ such that  $(u'_H)^p = f s^{p\sigma(u'_H)}$ in $H$. In particular, $|p|$ and $|q|$ are both at most $n$ because $|\sigma(u'_H)|$ and $|\sigma(u'_E)|$ both divide $n$.  
 
 Applying the triangle inequality to $(u'_H)^p = f s^{p\sigma(u'_H)}$ and  $z = f (u'_E)^q$ gives that 
 for suitable a constant  $C_{15} >0$, 
 \begin{align} 
 |f|_G & \leq 2  |p| |u'_H|_G  \leq C_{15}   n^2, \  \text{  and}     \label{eq length of f} \\    
 |z|_G & \leq |f|_G + |q| |u'_E|_G   \leq C_{15}   n^2.    \label{eq length of z}
\end{align}
 
By Lemma~\ref{lem:best w},  there exists a word $w$ such that $uw =wv$ in $G$ and 
$|w|_G \leq C_0  n^2$  and $\Sigma(w)   \leq  C_0  n$. 
Then $uw =wv \lambda^N$ in $\Lambda$ for some $N \in \Z$ and, in similar situations in previous cases  we have called on Proposition~\ref{prop: dist exp}, which would give that $|N| \leq C_{16}^{n^2}$ for some constant $C_{16} >1$.  However,  we need a tighter bound and for that we look to Lemma~\ref{lem: dist}.  Let $W = v^{-1} w^{-1} u w$.   Observe that for any words $\tau$ and $\pi$, we have $\Sigma(\tau \pi) \leq \Sigma(\tau) + \Sigma(\pi)$ and $\Sigma(\tau) = \Sigma(\tau^{-1})$. So $\Sigma(W) \leq 2\Sigma(w) + |u| + |v| \leq (2C_0 +1) n$ and with Lemma~\ref{lem: dist} we get
\begin{equation} \label{eq last bound on N}
|N| \leq |W| C^{\Sigma\left(\widehat{W}\right)} \leq C_{17}^{n}
\end{equation}
for a suitable constant $C_{17} >1$.

The set of all $W$ such that $u W = W v$ in $G$ is   $\set{ z^r w \mid r \in \Z}$.   

Now, $u$ commutes with $z$ in $G$, so $u z = z u \lambda^M$ in $\Lambda$  for some $M \in \Z$.  
  
By hypothesis, $\o{u} \o{W} = \o{W} \o{v}$ in $\Lambda$ for some word $ \o{W} \in \Lambda$.
Let $W$ be this $\o{W}$ with all letters $\lambda$ deleted.  Then  $\o{u}  W = W \o{v}$ in $\Lambda$ also and $u  W  = W   v$ in $G$, and so this $W$ must equal $z^r w$ in $G$ for some $r \in \Z$.  We   calculate that in $\Lambda$,   
\begin{align}
\o{u} W & = uW \lambda^{L_u} \label{eq e1} \\ 
& = u z^rw   \lambda^{L_u+K} \label{eq e2} \\ 
& = z^r u w   \lambda^{L_u+K+rM} \label{eq e3} \\ 
& = z^r  w v   \lambda^{L_u+K+rM +N} \label{eq e4} \\
& = W v  \lambda^{L_u+rM +N} \label{eq e5} \\
& = W \o{v}  \lambda^{L_u - L_v+rM +N} \label{eq e6} 
\end{align}
because \eqref{eq e1} $\o{u} = u \lambda^{L_u}$, \eqref{eq e2} and \eqref{eq e5} $W = z^rw \lambda^K$ for some $K\in \Z$,  \eqref{eq e3} $uz=zu \lambda^M$,  \eqref{eq e4} $uw=wv \lambda^N$,  and  \eqref{eq e6}  $\o{v} = v \lambda^{L_v}$.
  So $L_v = L_u + rM + N$.

If $M=0$, then the above calculation gives that  $\o{u}  W = W \o{v}$ in $\Lambda$  for  $W = z^rw$, irrespective of the value of $r$.   So we can take $r=0$ and we get that $\o{u}  w = w \o{v}$ in $\Lambda$.

If $M \neq 0$, then   $r = (L_v - L_u -N)/M$ and because of $|L_u| + |L_v| \leq n$ and \eqref{eq last bound on N}, we get that $|r| \leq C_{18}^n$   for a suitable constant $C_{18} >0$.

In either case, $|W|_G \leq |r| \cdot |z|_G + |w|_G$ and because  of $|w|_G \leq C_0   n^2$ and  \eqref{eq length of z},   $|W|_G  \leq C_{19}^n$ for a suitable constant $C_{19}>0$.

This completes our proof that $\CL_{\Lambda}(n) \preceq 2^n$.

\section{Why $\CL_{\Lambda}(n) \succeq 2^n$} \label{s: CL Lambda at least exp} 

As we saw in Section~\ref{s: distortion}, if we define $f(n)$ to be the exponent sum of the letters in $\phi^n\mleft(a_i\mright)$, 
for a suitable choice of $i$, then   $f(n)  \simeq 2^n$ and $[t, s^{-n} a_i s^n]$ is a word of length $4n+4$ that equals $\lambda^{f(n)}$ in $\Lambda$.   Without loss of generality we may assume $i=1$.  So, if for $n \geq 1$ we define $\overline{u_n} =a_1 t \lambda^{f(n)}$ and $\overline{v_n} = a_1 t $, then $|\o{u_n}|_{\Lambda} + |\o{v_n}|_{\Lambda} \leq 4n + 8$ and $\o{u_n} \sim \o{v_n}$ in $\Lambda$ because $$\o{u_n} t^{f(n)} =   a_1 t^{f(n)+1}  \lambda^{f(n)} = t^{f(n)} a_1 t = t^{f(n)} \o{v_n}.$$

The images $u_n$ and $v_n$ in $G$ of $\o{u_n}$ and $\o{v_n}$, respectively, are $u_n = v_n = a_1 t$.
A shortest word $w$  such that $\o{u_n} w  = w \o{v_n}$  in $\Lambda$ will contain no letters $\lambda$, because $\lambda$ is central, and will satisfy $u_n w  = w v_n$ in $G$.  

Case~\eqref{l hyp:b2} of Lemma~\ref{l: Centralizers in hyp case} tells us that  $Z_G \left( a_1 t \right) =  \langle a_1 \rangle  \times \langle t \rangle$.  So, if $u_n w  = w v_n$  in $G$, then $w = a_1^p t^q$ for some $p, q \in \Z$.   

However not all such $w$ will satisfy $\o{u_n} w  = w \o{v_n}$  in $\Lambda$.  Indeed,  $\o{u_n} w = a_1 t \lambda^{f(n)} a_1^p t^q = a_1^{p+1} t^{q+1} \lambda^{f(n)+p}$ and   $w \o{v_n} = a_1^p t^q  a_1 t  = a_1^{p+1} t^{q+1}   \lambda^{q}$.  So the $w= a_1^{p} t^{q}$ such that $\o{u_n} w  = w \o{v_n}$  in $\Lambda$ and those of the form $w = a_1^{p} t^{p + f(n)}$ for some $p \in \Z$.

For such $w$,  
 \begin{equation} \label{eq: length at least p + } 
 |w|_{\Lambda} \geq |p + f(n)|
 \end{equation} 
  because killing all generators other than $t$ maps $\Lambda \onto \langle t \rangle \cong \Z$ and $w \mapsto t^{p + f(n)}$.  Now,  $\langle a_1 \rangle$ is an infinite cyclic subgroup of a hyperbolic group $H$ and as such is an undistorted subgroup of $H$. On account of this and that killing $t$ and $\lambda$ maps $\Lambda \onto H$, there exists a constant $C>0$ such that 
   \begin{equation} \label{eq: length at least Cp} 
  |w|_{\Lambda} \geq | a_1^p |_H \geq C|p|.  
  \end{equation}
  So, if $|p| < f(n)/2$ then  $|w|_{\Lambda} \geq f(n)/2$ by virtue of \eqref{eq: length at least p + }, and otherwise  $|w|_{\Lambda} \geq C f(n)/2$ by virtue of \eqref{eq: length at least Cp}.
  That $\CL_{\Lambda}\mleft(n\mright)   \succeq 2^n$ follows.  
  
  This completes our proof of Theorem $1'$.

\section{Fibre products: examples further up the Grzegorczyk hierarchy} \label{s:Fibre products}

We now turn our attention to a fibre product construction that
yields finitely presented groups displaying a wide range of conjugator length functions. Our purpose here is two-fold: first,
we want to construct specific finitely presented groups with 
large but computable conjugator length functions including representatives comparable to every level of the Grzegorczyk hierarchy of primitive recursive functions
 (Theorem \ref{t:big fellas}); secondly, 
we want to describe a framework (Remark~\ref{rem:frontier}) that holds the potential to provide 
calculations of $\CL\mleft(n\mright)$ for finitely presented groups in great generality.

\def\ssm{\smallsetminus}

The groups that we shall construct to prove Theorem \ref{t:big fellas} are
obtained by following the general template
for constructing ``designer groups" described in \cite{BridsonICM},
which is based on refinements 
of the Rips Construction \cite{rips} and the 1-2-3 Theorem 
of Baumslag, Bridson, Miller and Short \cite{BBMS2}. 
As is usual with this template,  we will have to craft input groups 
carefully to achieve the desired output group.

We remind the reader that
a group $G$ is said to be of {\em type ${\rm{F}}_3$} if it has a classifying
space $K(G,1)$ whose 3-skeleton is finite.

The Rips
Construction \cite{rips}
associates to any finite group-presentation $\mathcal{Q}$ a short exact sequence
$$
1\to N\to \G \overset{p}\to Q\to 1
$$
where $Q$ is the group presented by $\mathcal{Q}$, while
$N$ is a 2-generator group and $G$ is a torsion-free hyperbolic group
that satisfies a prescribed small-cancellation condition.
The  1-2-3~Theorem \cite{BBMS2}
implies that if $Q$ is of type ${\rm{F}}_3$ then the
{\em fibre product}
$$
P:= \set{(x,y)\in \G \times \G \mid p(x)=p(y) }
$$
is finitely presented.  
Moreover, it was proved in \cite{BHMS} that there is an algorithm that, 
given the presentation $\mathcal{Q}$ and a
set of $\Z Q$-module generators
for the second homotopy module $\pi_2\mathcal{Q}$ (or, equivalently, a combinatorial
model for the 3-skeleton of $K(Q,1)$),  will output a finite presentation for $P$; 
if the presentation is aspherical, then the algorithm simplifies considerably.

A primary goal of \cite{BBMS2} was to construct a finitely presented subgroup of a product of hyperbolic groups  
 such that   the membership and conjugacy problems for the subgroup were unsolvable.  
This was done by combining the Rips Construction and the 1-2-3 Theorem as described above: if $Q$
has an unsolvable word problem, then the fibre product $P<\G\times\G$ has the desired properties.  
This construction builds on an idea that originates in the work of Mihailova \cite{mihailova2}, who considered the
case where $\G$ is a free group. 
The following basic lemma from her work contains the key facts and the reader who is new to these ideas will find the exercise of 
proving it to be instructive.

\begin{lemma} [Mihailova] Let $Q = \langle a_1, \ldots, a_n \mid r_1, \ldots, r_m \rangle$, let $F$ be the
free group $F(a_1,\dots,a_n)$, and let
$$
P = \{ (u,v) \mid u=v \text{  in } Q\} < F\times F.
$$
Then,  $P$ is generated by $ \{(a_1, a_1), \ldots, (a_n,a_n), (r_1,1), \ldots, (r_m,1)\}$.
And for all $w\in F$,
\begin{enumerate}
\item $w = 1$ in $Q$ if and only if  $(w,1) \in P$, and 
\item provided $r \in F$ is not a proper power,     $(wrw^{-1},r)  \sim  (r,r)$ in $P$ if and only if $w=1$ in $Q$.
\end{enumerate}
\end{lemma}

By refining the proofs of such basic results,  one can establish 
quantified relationships between membership and conjugacy problems, on the one hand, and 
word problems, on the other:
 loosely speaking, the Dehn function  for $Q$ is reflected in the distortion function of 
$P<\G\times \G$,  
and these functions 
together with the geometry of cyclic subgroups in $Q$ account for  the difficulty of the conjugacy problem in $P$. 
These relationships are explored in detail in \cite{Bridson13} and \cite{Bridson14}. In particular,   the following close relationship  between the Dehn function of $Q$ and the conjugator length function of $P$ is established in \cite{Bridson13}, under the hypothesis that $Q$ has \emph{uniformly quasigeodesic cyclics} (UQC), meaning that there is a constant $\lambda >0$
such that $|q^n|_Q\ge \lambda |n|$ for all $n\in\Z$ and $q\in Q\ssm\{1\}$.

\begin{theorem}\label{r:input}\cite[Corollary~C]{Bridson13}
Let $P<\G\times \G$ be the fibre product associated
to an epimorphism $\G\to Q$ where $\G$ is a torsion-free hyperbolic group and
$Q$ is a finitely presented group that has uniformly quasigeodesic cyclics. Then  
$$\Dehn_Q\mleft(n\mright)\preceq \CL_P\mleft(n\mright)\preceq \Dehn_Q\mleft(n^2\mright).$$
\end{theorem}

To identify groups $Q$ meeting the hypothesis of this theorem we look to: 
 
\begin{lemma}[see Section~4.1 of \cite{Bridson13}] \label{l:dot00} { \ } 
\begin{enumerate}
\vspace*{-2mm}
\item Torsion-free hyperbolic groups and \textup{CAT}$(0)$ groups have the UQC property. 
\item If  $Q$ is a finitely generated group with UQC, then any HNN extension of 
the form $Q\dot\ast_M$ has UQC.
\end{enumerate}
\end{lemma}

(Here $Q\dot\ast_M$ denotes the trivial HNN-extension $\langle Q, t \mid [t,m] =1 \ \forall m \in M  \rangle$ of $Q$ along a subgroup $M \leq Q$.)

Our main interest here lies with large Dehn functions
associated to groups that have useful additional properties. We
would like to have concise presentations for them, as well as geometry to aid
our understanding.
  
We first consider  the function $\Delta\mleft(n\mright)$ defined 
recursively by $\Delta(1)=2$ and $\Delta(k+1)=2^{\Delta(k)}$. This 
function  
arises in several settings:  $\Delta(\lfloor   \log n\rfloor)$
is the Dehn function of Higman's group with no finite quotients \cite{higman}
$$
\langle a,b,c,d \mid a^{-1}ba=b^2,\, b^{-1}cb=b^2,\, c^{-1}dc=d^2,\,
d^{-1}ad=a^2 \rangle
$$
and of the Baumslag--Gersten one-relator group \cite{Baumslag2}
\begin{equation} \label{eq BG}
\langle x, t \mid (t^{-1}x^{-1}t)x(t^{-1}xt)=x^2 \rangle;	
\end{equation}
see \cite{Bridson15} and \cite{Platonov}.  Further, $\Delta\mleft(n\mright) = A_3\mleft(n\mright)$  the third \emph{Ackermann function}---see Section~\ref{intro}.  In fact,  all of the   Ackermann functions $n \mapsto A_k\mleft(n\mright)$ arise as Dehn functions of groups with small aspherical presentations. 
 These Ackermann functions play a central role in the work of
Dison and Riley on hydra groups \cite{DR}.
They prove that $A_k$ is the Dehn function of 
the HNN extension $Q_k = H_k\dot\ast_{L_k}$ where $H_k$ is the free-by-cyclic
 group $\<a_1,\dots,a_k,t  \mid t^{-1} a_1 t =a_1, \, t^{-1}a_i t = a_i a_{i-1} \ (i >1)  \>$  and $L_k$ is the free subgroup generated
$\{a_1t,\dots,a_kt\}$:
\begin{equation}  \label{e:hydra}
Q_k = \langle a_1,\dots,a_k,t, s \mid t^{-1} a_1 t =a_1, \, t^{-1}a_i t = a_i a_{i-1} \ (i >1); [s,a_it]=1 \  (i>0) \rangle. 
\end{equation}
\medskip

\begin{proof}[Proof of Theorem~\ref{t:big fellas}] We want finitely presented groups whose conjugator length functions  are comparable to fast-growing functions at all levels of the Grzegorczyk hierarchy.
Examples with conjugator length equivalent to $2n =A_1\mleft(n\mright)$ are easy to come by, and   in Theorem~\ref{t: exponential CL}
we constructed examples  growing like $2^n = A_2\mleft(n\mright)$. So we now 
assume that $k \geq 3$, and we want to exhibit
a finitely presented group $P_k$ such that $ A_k\mleft(n\mright)\preceq \CL_{P_k}\mleft(n\mright) \preceq A_k\mleft(n^2\mright)$.

Let $\Gamma_k$ be the group described above.  It is obtained from a finitely generated free group $F_k = F\mleft(a_1, \ldots, a_k\mright)$ by taking two HNN extensions,
each along finitely generated free subgroups. A standard topological argument
shows that the standard 2-complex of the natural presentation for any
such extension is aspherical, so the group is of type $\rm{F}_3$.
The Rips construction provides us with a torsion-free hyperbolic group $\G_k$
and an epimorphism $\G_k \onto Q_k$ with finitely generated kernel, and the
1-2-3 Theorem~\cite{BBMS2} tells us that the associated fibre product
$P_k < \G_k \times \G_k$ is finitely presented.

Moreover, $Q_k$ is of the form $H_k \dot\ast_{L_k}$, and it is shown in \cite{DR}  that $H_k$ is the fundamental
group of a compact non-positively curved space. Lemma~\ref{l:dot00} 
tells us that such groups have uniformly quasigeodesic cyclics, so 
Theorem~\ref{r:input} applies and we conclude that 
$ A_k\mleft(n\mright)\preceq \CL_{P_k}\mleft(n\mright) \preceq A_k(n^2)$.\end{proof}

\begin{remark} 
Recently, \cite{Gillis2} Gillis showed that the conjugator length function of the Baumslag--Gersten one-relator group \eqref{eq BG} is fast-growing.   He gave upper and lower bounds which differ but both take the form of logarithmic-height towers of exponential functions. 
\end{remark}

\begin{remark}\label{rem:frontier}  Theorem~\ref{r:input} has an antecedent in \cite{Bridson13}  which says that for $Q$ a finitely presented group, and $\G \onto Q$ an epimorphism from a torsion-free hyperbolic group, and $P < \G \times \G$ the associated fibre product, 
$$\Dehn_Q\mleft(n\mright) \preceq  \CL_P \mleft(n\mright) \preceq  \CL_P^{\G\times \G}\mleft(n\mright) \preceq \Dehn^c_Q\mleft(n\mright),$$ where
$\CL^{G\times G}_P \mleft(n\mright)$
 is quantified over $u,v$ with
$|u|_{G\times G}+|v|_{G\times G}\le n$ rather than 
$|u|_P+|v|_P\le n$ and
$\Dehn_Q^c\mleft(n\mright):\N\to \N$ is the {\em rel-cyclics Dehn function},
\begin{equation}
\Dehn^c\mleft(n\mright) : = \max_{w,u} \set{ {\rm{Area}}(w\,u^p) \,  \colon \, |w|+|u|\le n,\ w=_Qu^{-p},\
|p|< o(u)/2}.
\end{equation} 
This relationship motivates us to consider
which functions arise as rel-cyclic Dehn functions $\Dehn^c\mleft(n\mright)$
 of groups of type ${\rm{F}}_3$.
This question seems amenable to
attack using the many techniques developed to study Dehn functions
and subgroup distortion, and this expectation
lends weight to the conviction that the set of $\simeq$ classes of
conjugator length functions of finitely presented
groups is likely to be as diverse as the set of 
Dehn functions.\footnote{Indeed, since we expressed this conviction in the first version of this article, Gillis and Wagner [GiWa] have confirmed its truth with definitive results based on the technology of S-machines.}

\end{remark}

\bibliographystyle{alpha}
%\bibliography{bibli}
\bibliography{$HOME/Dropbox/Bibliographies/bibli}

\ni Martin R.\ Bridson, 
Mathematical Institute, Andrew Wiles Building, Oxford OX2~6GG, United Kingdom, {bridson@maths.ox.ac.uk},  
\href{https://people.maths.ox.ac.uk/~bridson/}{people.maths.ox.ac.uk/bridson/}

\ni  {Timothy R.\ Riley}, Department of Mathematics, 310 Malott Hall,  Cornell University, Ithaca, NY 14853, USA,  {tim.riley@math.cornell.edu}, \href{https://pi.math.cornell.edu/~riley/index.html}{math.cornell.edu/$\sim$riley/}

 \end{document}